\documentclass[10pt]{amsart}
\usepackage{amsmath}
\usepackage{amsthm}
\usepackage{amsfonts}
\usepackage{amssymb}
\usepackage{graphicx}

\theoremstyle{plain}
\newtheorem{thm}{Theorem}[subsection]
\newtheorem{prop}{Proposition}[subsection]
\newtheorem{remk}{Remark}[subsection]
\newtheorem{defi}{Definition}[subsection]
\newtheorem{exam}{Example}[subsection]
\newtheorem{lemma}{Lemma}[subsection]
\newtheorem{cor}{Corollary}[subsection]

\begin{document}
\title{Codes for Square-Tiled Surfaces}
\author{Kuo-Chiang Tan}
\address{Department of mathematic, Indiana University, Bloomington IN 47405, USA}
\email{kutan@indiana.edu}

\thanks{}


\keywords{Square-tiled surfaces, Veech group}

\date{}
\maketitle
\begin{abstract}
In this paper,  we use permutation elements to record cylinder decompositions of a square-tiled surface $X$. Collecting  all such possible permutation elements that record cylinder decompositions, we can enumerate the $SL_2(\mathbb{Z})$ orbit of a given surface $X$ and give a method to determine whether or not a matrix $A\in SL_2(\mathbb{Z})$ is the differential of an affine diffeomorphism of $X$.
\end{abstract}

\section{Introduction}
 A square-tiled surface is a topological surface obtained by gluing together n squares labeled $\{1,2,\ldots,n\}$ in a horizontal and vertical fashion. The horizontal (resp. vertical) gluings are recorded by a permutation $\sigma\in S_n$ (resp. $\tau\in S_n$). (See $\S 2$ for more details.) Square-tiled surfaces were introduced in W. Thurston's work [17] and can be regarded as ``rational points" in the moduli space of the moduli space of all 1-forms that are holomorphic with respect to some complex structure.  Overviews and further properties can be found, e.g., in [1], [3], [4], [5], [7], [8], [11], [16] and [18].

 This paper is devoted to the study of the cylinder decompositions of a given square-tiled surface $X=X(\sigma,\tau)$ as well its Veech group, $Veech(X)$ which is the subgroup of $SL_2(\mathbb{R})$ consisting of all differentials of affine diffeomorphisms and denoted by $Veech(X)$. Using the permutations $\sigma$, $\tau$ and cutting sequences of saddle connections, we encode the combinatorics of the saddle connections (resp. cylinder decomposition) in the direction $(a,b)$ with an element $\mathsf{Code}(a,b)\in S_n$ (resp. $\mathsf{Code}^{L}(a,b)\in S_n$). The permutation  $\mathsf{Code}^{L}(a,b)$ is called  the left code of the cylinder decomposition of $X$ in the direction $(a,b).$  Given a $2\times 2$ matrix $A$ with non negative integer entries and determinant 1, we define $\mathsf{Code}^{L}(A)$ to be the pair of left codes of its column vectors. Given such two matrices $A$ and $B$, we say the $\mathsf{Code}^L(A)$ is conjugate to $\mathsf{Code}^{L}(B)$ if and only if there exists  $\omega\in S_n$ so that the left code of the first (resp. second) column of A is conjugate to the left code of the first (resp. second) column of B by $\omega.$ Then the following theorem is one of our main results:

\begin{enumerate}
\item[]{\textit{\textbf{Theorem 6.1.1.}}}
\textit{ A $2\times 2$ matrix $A$  with non negative integer entries and determinant 1 belongs to Veech($X$) if and only if $\mathsf{Code}^L(A)$ is conjugate to $\mathsf{Code}^L(I)$ .}
\end{enumerate}
For the collection of all surfaces with a translation structure, there is an $SL_2(\mathbb{Z})$ action on this collection (see $\S$2).
Let $S^+(X)$ be the set
\begin{align*}\{X(\omega_1,\omega_2)|(\omega_1,\omega_2)=&\mathsf{Code}^L(A),\mbox{ for some matrix A}\\
&\mbox{ with non negative entries and determinant 1}\}.\end{align*} Then we have:

\begin{enumerate}
\item[] {\textit{\textbf{Theorem 6.2.1.}}}
\textit{Let $X$ be a square-tiled surface, then the index of the Veech group in $SL_2(\mathbb{Z})$ is $|S^+(X)|.$}
\end{enumerate}

The proof of theorem 6.1.1 is based on the following idea: for any $A=[a_{ij}]$ in Theorem 6.1.1, the saddle connections whose directions are given by the first and second column vectors induce a tiling of the surface by the parallelogram spanned by these column vectors.  The left codes of the vectors
 record how we tiled these parallelograms to get the surface $X$. If these left codes are conjugate to the left codes of $(1,0)$ and $(0,1)$ by a permutation $\omega\in S_n,$ we can define an affine diffeomorphism of $X$ with the derivative $A.$ This proves the theorem 5.1.1.

To prove Theorem 6.2.1, since $SL_2(\mathbb{Z})$ can be generated by finitely many matrices with non negative entries and there is an $SL_2^+(\mathbb{Z})$ action on $S^+(X)$, we can conclude that $S^+(X)$ equals the $SL_2(\mathbb{Z})$ orbit of $X$.

We begin this paper by giving the background that we need in $\S 2$. In $\S$3, we define the cutting sequences and codes of saddle connections and discuss the relation between them; in the rest of this section, we define the left codes to record cylinder decompositions on $X$. In $\S$4, we briefly discuss the algebraic properties of cutting sequence and the left codes. In $\S$4.1, we apply the property called Farey addition to the left codes and this property makes left codes easy and systematical to compute. The relations between slopes and left codes will be discussed in $\S$4.2 and $\S$4.3. In $\S$5.1, the closed system of $X$, the main tool in this paper, is introduced and we also discuss some basic algebraic properties of the closed systems.
In $\S$5.2, we will study the $SL_2^+(\mathbb{Z})$( the collection of all matrices with non negative entries and determinant 1) action on closed system. In $\S$5.3, we consider the set consisting of all left codes of matrices in $Veech(X)\cap SL_2^{+}(\mathbb{Z}).$  Applying the results in $\S$5.2, we give a group structure on this set and it is isomorphic to $G_X/S_X$, where $G_X=\{\omega\in S_n|\mbox{ }\omega\Omega_X\omega^{-1}=\Omega_X\}$ and
$S_X=\{\omega\in S_n|\mbox{ }\omega\mathsf{Code}^L(I)\omega^{-1}=\mathsf{Code}^L(I)\}.$  In $\S$6, we end this paper with the applications of closed systems. Applying the closed systems, we obtain the main results, Theorem 6.1.1 and 6.1.2.

\section{Background}
Let us start this paper by introducing the followings: \\

\textbf{Permutation groups.} Let $S_n$ be the collection of injective functions of the set $\{1,2,\ldots,n\}$ and for elements $\sigma$ and $\tau$ in $S_n$, we defined the element $\sigma\cdot\tau$ by applying $\sigma$ first and then applying $\tau.$ $S_n$ together with this operation forms a group and is group isomorphic to the group $(S_n,\circ)$ whose group operation is composition of functions. In this paper, we call $(S_n,\cdot)$ a permutation group and any element in it is called a permutation. For convenience, the element $\sigma\tau$ means the permutation $\sigma\cdot\tau.$\\

\textbf{Square-tiled surfaces $X(\sigma,\tau)$.} Let $\{Q_i|i=1,\ldots ,n\}$ be n's unit squares and $\sigma$ and $\tau$ be elements in the permutation group $S_n$. The square-tiled surface $X=X(\sigma,\tau)$ is the surface obtained by gluing the right (resp. top) sides of the square $Q_i$ to the left (resp. bottom) side of $Q_{\sigma(i)}$ (resp. $Q_{\tau(i)}$). Two square-tilde surfaces $X(\sigma,\tau)$ and $X(\sigma',\tau')$ are equivalent if $\sigma$ and $\tau$ is conjugate to $\sigma'$ and $\tau'$ by a permutation $\omega$ respectively.

  Since there is a canonical branch covering from $X$ to a unit square, if $Z={p_1,p_2,\ldots ,p_k}$ is the collection of all branch points of $\pi,$ then
  \begin{enumerate}
  \item There is a translation structure on $X\setminus Z$, i.e., the transition functions are translations.
  \item a holomorphic 1-form $\omega=\pi^\ast dz$ defined on $X$ with the zero set $Z$
  \item $\omega$ locally can be expressed by $dz$ on $X\setminus Z$ and $z^{n_{i}}dz$( $n_i>0$) on a neighborhood of each zero $p_i$.
  \end{enumerate}
    We call each $p_i$ a critical point of $X$ with the cone angle $n_i+1>1$.  A point on $X$ is called marked point if it is identified by the corner of some $Q_i.$  Any critical point is a marked point. Suppose $\{q_i|i=1,\ldots,n\}$ are the collection of all southwest corner of the square $Q_i$ and $P$ is the projection from squares $Q_i$ to $X$. Consider the permutation $\Theta:=\sigma^{-1}\tau^{-1}\sigma\tau=\Theta_1\Theta_2\ldots \Theta_m$ where $\Theta_1,\Theta_2,\ldots,\Theta_m$ are disjoint cycles, then $P(q_i)$ is a critical point if and only  $\Theta_k(i)\neq i$ for some $k$. Moreover, its cone angle is ord$(\Theta_k).$\\

\textbf{Saddle connections.} Saddle connections are the "straight" line segments on square-tiled surfaces whose end points are marked points. Since there is a translation structure on $X$, we can defined the slope and the direction of any line segment. For any saddle connection, its slope is (extended) rational number $\frac{p}{q}$ and the direction is $(q,p)$ or $(-q,-p)$ ( the integers $q$ and $p$ are either coprime or the pair $(q,p)$ is in $\{\pm(1,0),\pm(0,1),(0,0)\}$.).\\

\textbf{Cylinder decompositions of square-tiled surfaces.} Let $X=X(\sigma,\tau)$ be a given square-tiled surface and $m$ be any extended rational number. Then the compliment of of all saddle connections with slopes $m$ is the disjoint union of Euclidean open cylinders. For the horizontal( resp. vertical) direction, the number of the cylinders in its cylinder decomposition is the number of disjoint cycles of $\sigma$( resp. $\tau$).\\

\textbf{Affine diffeomophisms and Veech groups.} Let $X$ be a square-tiled surface. An affine diffeomorphism of $X$ is an orientation preserving homeomorphism $f$ on $X$ such that
\begin{enumerate}
\item $f$ leaves the zero set of $dz$ invarient.
\item $f$ is diffeomorphic on the surface $(X\setminus\{p_1,p_2,\ldots,p_k\}$ and its derivative is an matrix is $SL_2(\mathbb{Z})$ (with respect to the translate structure induced by $\omega$).
\end{enumerate}
For any square-tiled surface $X$, there are two affine diffeomorphisms are known: the horizontal and vertical Dehn twists. These are the Dehn twists with respect to the horizontal and vertical cylinder decompositions.

 The collection of the derivatives of affine diffeomorphisms on $X$ is called the Veech group of $X$ and denoted by Veech($X$). The group is not trivial since the derivative of horizontal and vertical Dehn twists on $X$ are of the forms:
\[\left(
    \begin{array}{cc}
      1 & h \\
      0 & 1 \\
    \end{array}
  \right),
  \left(
    \begin{array}{cc}
      1 & 0 \\
      v & 1 \\
    \end{array}
  \right)
\] where $h$ and $v$ are the orders of permutations $\sigma$ and $\tau$, respectively.\\

\textbf{Action of $SL_2(\mathbb{Z})$ on translation surfaces.}
A surface is called a translation surface if there is a translate structure $\{(U,f_u)\}$ defined on $X\setminus P$ where $P$ is a finite set. For any element $A\in SL_2(\mathbb{Z})$, the action $A\{(U,f_u)\}$ is defined by $\{(U,A\circ f_u)\}$. This gives a new translate structure on the surface $X\setminus P$. For a square-tiled surface $X$, Veech($X$) is the stabilizer of the surface $X$ for the action of $SL_2(\mathbb{Z}).$

\section{Cutting sequences and codes }

   Let $X=X(\sigma,\tau)$ be a square-tiled surface tiled by $n'$s squares $Q_i$, $i\in\{1,2,\ldots,n\}.$ We say $\frac{p}{q}$ is a rational number if $(q,p)$ is either a coprime pair of positive integers or in the set $\{(0,1),(1,0)\}.$  Let $\left(\frac{p}{q}\right)$ be a line segment in $\mathbb{R}^2$ with end points $(0,0)$ and $(q,p)$ and its slope is the rational number $\frac{p}{q}$.

\subsection{Cutting sequences and codes of saddle connections}
Assume the rational number $\frac{p}{q}$ is not 0,1 and $\infty$. Define $w_0:=(0,0)$ and $w_i=(x_i,y_i)$ is the i-th point on $\left(\frac{p}{q}\right)$ in the direction $(q,p)$ such that either $x_i$ or $y_i$ is an integer.  We define a function $f:\{w_i:i=1,\ldots,{p+q-2}\}\rightarrow\{x,y\}$ by $f(w_i)=x$ ($y$) if $x_i$ ($y_i$) is an integer.
\begin{defi}
 For any rational number $\frac{p}{q}$, the cutting sequence of line segment $\left(\frac{p}{q}\right)$ is the word defined by
 $$\mathsf{Cut}\left(\frac{p}{q}\right):=\left\{
                                           \begin{array}{llll}
                                             f(w_1)f(w_2)\ldots f(w_{p+q-2}) & \hbox{r} \notin\{0,1,\infty\} ; \\
                                             x^{-1} & r=0;\\
                                             e & r=1;\\
                                             y^{-1} & r=\infty.
                                           \end{array}
                                         \right.
$$

\end{defi}

For convenience, the notation $x^n$ means $f(w_i)=x$ for some $i=j,\ldots,j+n-1$. For instance, $\mathsf{Cut}\left(\frac{2}{5}\right)=xxyxx=x^2yx^2$ .
\begin{figure}[h]
\centering
\includegraphics[width=0.6\textwidth]{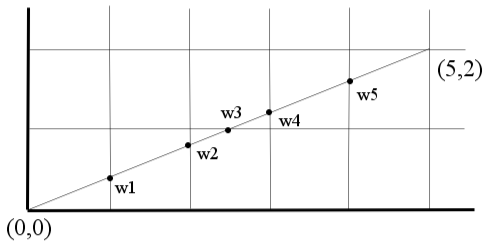}
\caption{The $\mathsf{Cut}\left(\frac{2}{5}\right)$ is equal to $xxyxx=x^2yx^2.$}
\end{figure}

Let $\frac{p}{q}$ be a rational number which is not 0 and $\infty.$ The saddle connections $\{\gamma_i\mid i=1,\ldots, n\}$ with the slope $\frac{p}{q}$ induce an element $\omega$ in $S_n$. It is defined by the following: assume the end points of $\gamma_i$ are the south-west corner of the square $Q_i$ and the north-east corner of some $Q_j$, then $\omega(i)=j.$

\begin{defi}
Let $r$ be any rational number. The code of the saddle connections with the slope $r$ is defined by: $$\mathsf{Code}_X(r):=\left\{
                                                          \begin{array}{ll}
                                                            \omega & r\notin\{0,\infty\} \\
                                                            \sigma^{-1} & r=0\\
                                                            \tau^{-1} & r=\infty.
                                                          \end{array}
                                                        \right.
$$ If there is no confusion about the surface that we are discussing, we denote the code of the saddle connection with the slope $\frac{p}{q}$ by $\mathsf{Code}(\frac{p}{q})$.
\end{defi}
The following proposition gives us a way to get the
$\mathsf{Code}\left(\frac{p}{q}\right)$ via $\mathsf{Cut}\left(\frac{p}{q}\right):$
\begin{prop} Suppose $\mathsf{Cut}\left(\frac{p}{q}\right)=x^{n_1}y^{m_1}x^{n_2}y^{m_2}\ldots{x^{n_k}y^{m_k}}$, $n_i,m_j>{0.}$ Then for any square-tiled surface $X=X(\sigma,\tau)$, $\mathsf{Code}\left(\frac{p}{q}\right)$ is equal to $$\sigma^{n_1}\tau^{m_1}\sigma^{n_2}\tau^{m_2}\ldots{\sigma^{n_k}\tau^{m_k}}.$$
\end{prop}
\begin{exam} Consider the square-tiled surface $X=X((1,2),(1,3))$. The following figure shows three saddle connections with slope $\frac{2}{5}$ on $X$. It indicates that $\mathsf{Code}\left(\frac{2}{5}\right)=(1,3).$
\begin{figure}[h]
\centering
\includegraphics[width=0.5\textwidth]{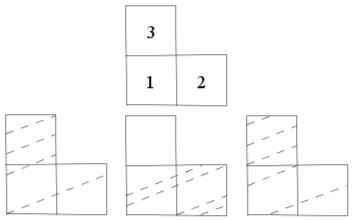}
\caption{Let $X=X((1,2),(1,3)).$ Since $\mathsf{Cut\left(\frac{2}{5}\right)}=x^2yx^2$, the left code $\mathsf{Code}\left(\frac{2}{5}\right)$ equals $(1,2)^2(1,3)(1,2)^2=(1,3).$}
\end{figure}
\end{exam}

\subsection{The left and right codes of  cylinder decompositions}
In this section, we want to introduce the left and right codes. These are permutations and used to record cylinder decompositions on a square-tiled surface with respect to rational slopes.

It is well-known that for any square-tiled surface $X=X(\sigma,\tau)$, the complement of saddle connections $\{\gamma_i\mid i=1,\ldots,n\}$ with a rational slope $r=\frac{p}{q}$ is a disjoint union of open Euclidean cylinders $\{C_i\mid i=1,\ldots,m_r\}$. This is called the cylinder decomposition on $X$ w.r.t. the saddle connections with the slope $r.$ Moreover, all marked points of $X$ are contained in the boundaries of these cylinders. A saddle connection $\gamma_j$ is called a left boundary of some cylinder $C_i$, if the cylinder $C_i$ is at our left hand side when we are walking along  $\gamma_j$ in the direction $(q,p)$.

Pick a cylinder $C_j$ and assume that $\{\gamma_{i_{k}}\mid k=1,\ldots,n_j\}\subset \{\gamma_i\mid i=1,\ldots,n\}$ is the collection of all left boundaries of $C_j.$ Moreover, suppose that if we are walking on $\gamma_{i_k}$ in the direction $(q,p)$, the next left boundary of $C_j$ that we will meet is $\gamma_{i_{k+1}}.$ (For $k=n_j$, the next left boundary is $\gamma_{i_1}$.)  Let the south-west corner of the square $Q_{i_k}$ is an end point of $\gamma_{i_k}$, then this induces an element $\omega_j\in S_n$ defined by $\left(i_1,i_2,\ldots,i_{n_j}\right)$.
\begin{defi}
Let $X=X(\sigma.\tau)$ be a square-tiled surface and $r$ be a rational number. The left code of the saddle connections with the slope $r$ is defined by $$\mathsf{Code}_X^L(r)=\omega_1\omega_2\ldots\omega_{m_r}.$$
\end{defi}
\begin{remk} About the definition 3.2.1, somethings have to be mentioned:
\begin{enumerate}
\item Suppose $\mathsf{Code}_X^L\left(\frac{p}{q}\right)=\omega_1\omega_2\ldots\omega_k$ where $\omega_j's$ are disjoint $n_j$-cycle, then each $\omega_j$ associates a cylinder $C_j$ in the cylinder decomposition w.r.t the saddle connections with the slope $\frac{p}{q}$ and the surface area of $C_j$ is the order of $\omega_j.$
\item The right boundaries of a cylinder  and the right code of a cylinder decomposition also can be defined by the similar ways.
\item If there is not any confusion about the surface we are talking about, we use the notation $\mathsf{Code}^L\left(r\right)$ to denote the left code of the saddle connections with slope $r$.
\end{enumerate}
\end{remk}
\begin{exam}

Let $X=X((1,2),(1,3))$ and the slope of the saddle connection be $\frac{2}{5}$. Then the left code is $\mathsf{Code}^L(\frac{2}{5})=(1,2).$
    \begin{figure}[h]
    \centering
    \includegraphics[width=0.2\textwidth]{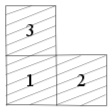}
    \caption{Let $X=X((1,2),(1,3)).$ For the saddle connections with slope $\frac{2}{5},$ its left code is $\mathsf{Code}^L\left(\frac{2}{5}\right)=(1,2)$ and the cylinder decomposition contains two cylinders with area 1 and 2.}
    \end{figure}
\end{exam}
There is an obvious relation between the code and left code:
\begin{prop}For a surface $X=X(\sigma,\tau)$, we have $$\mathsf{Code}^L\left(\frac{p}{q}\right)=\mathsf{Code}\left(\frac{p}{q}\right)\tau\sigma.$$
\end{prop}

\begin{remk} For the right code, we also have the following relation:
$$\mathsf{Code}^R\left(\frac{p}{q}\right)=\mathsf{Code}\left(\frac{p}{q}\right)\sigma\tau.$$ This also gives a relation between the left and right codes:
$$\mathsf{Code}^R\left(\frac{p}{q}\right)=\mathsf{Code}^L\left(\frac{p}{q}\right)\sigma^{-1}\tau^{-1}\sigma\tau=\mathsf{Code}^L\left(\frac{p}{q}\right)\Theta.$$
\end{remk}

\section{Algebraic properties of $\mathsf{Cut}$ and $\mathsf{Code^L}$}

\subsection{The additivity of the cutting sequences}
\begin{defi}
For any two rational numbers $\frac{p_1}{q_1}$ and $\frac{p_2}{q_2}$, the notation $$\frac{p_1}{q_1}<_n\frac{p_2}{q_2}$$ means $\frac{p_1}{q_1}<\frac{p_2}{q_2}$ and $p_2q_1-p_1q_2=1$. If  $\frac{p_1}{q_1}<_n\frac{p_2}{q_2}$, the Farey addition of these two rational numbers is defined by $$\frac{p_1}{q_1}\oplus\frac{p_2}{q_2}:=\frac{p_1+p_2}{q_1+q_2}.$$
In this case,  $\frac{p_1}{q_1}$ and $\frac{p_2}{q_2}$ are called the neighbors of $\frac{p_1}{q_1}\oplus\frac{p_2}{q_2}.$
\end{defi}
The main result in this section is the following theorem:
\begin{thm}
Let $\frac{p_1}{q_1}<_n\frac{p_2}{q_2}$,  then the cutting sequence of the Farey addition $\frac{p_1}{q_1}\oplus\frac{p_2}{q_2}$ is equal to    $$\mathsf{Cut}\left(\frac{p_1}{q_1}\oplus\frac{p_2}{q_2}\right)=\mathsf{Cut}\left(\frac{p_1}{q_1}\right)yx\mathsf{Cut}\left(\frac{p_2}{q_2}\right).$$
\end{thm}
\begin{proof}
Suppose that $\frac{p_1}{q_1}=\frac{0}{1},$ then $\frac{p_2}{q_2}$ should be $\frac{1}{n}$ for some $n\in\mathbb{N}.$ Since the cutting sequences of $\frac{0}{1}$ and $\frac{1}{n}$ are $y^{-1}$ and $x^{n-1}$ respectively, we have $\mathsf{Cut}\left(\frac{1}{n+1}\right)=\mathsf{Cut}\left(\frac{0}{1}\right)yx\mathsf{Cut}\left(\frac{1}{n}\right).$ Similarly, if $\frac{p_2}{q_2}=\frac{1}{0}$, we also have the result that $\mathsf{Cut}\left(\frac{n+1}{1}\right)=\mathsf{Cut}\left(\frac{n}{1}\right)yx\mathsf{Cut}\left(\frac{0}{1}\right).$

For $\frac{p_1}{q_1}$ and $\frac{p_2}{q_2}\notin \{\frac{0}{1},\frac{1}{0}\},$  the proof is based on the fact that the line segment $\left(\frac{p_1+p_2}{q_1+q_2}\right)$ is contained in the set $A=P^L(q_1,p_1)\cup  S(q_1,p_1)\cup \left(P^L(q_2,p_2)+(q_1,p_1)\right)$ where $S(q,p)$ is the unit square whose south-east corner is $(q,p)$ and $P^L(q,p)$ is the set bounded by $y=\frac{p}{q}x$, $y=\frac{p}{q}x+\frac{1}{q}$, $x=0$ and $y=p.$

Since there is not any integer pair contained in the interior of $A$, two line segments $L_1:=\left(\frac{p_1+p_2}{q_1+q_2}\right)\cap P^L(q_1,p_1)$ and $L_2:=\left(\frac{p_1}{q_1}\right)$ enjoy the same cutting sequence. This is also true for the line segments $R_1:=\left(\frac{p_1+p_2}{q_1+q_2}\right)\cap \left(P^L(q_2,p_2)+(q_1,p_1)\right)$ and $R_2:=\left(\frac{p_2}{q_2}\right)+(q_1,p_1).$ By the trivial fact that the cutting sequence of the line segment $\left(\frac{p_1+p_2}{q_1+q_2}\right)\cap S(q_1,p_1)$ is $yx,$ we prove the this theorem.
\end{proof}
\begin{remk} $\mbox{ }$
\begin{enumerate}
\item Suppose $\frac{p_1}{q_1}<_n\frac{p_2}{q_2}$, then $\frac{p_1}{q_1}<_n\left(\frac{p_1}{q_1}\oplus\frac{p_2}{q_2}\right)<_n\frac{p_2}{q_2}$.
\item By a similar argument, we also have   $$\mathsf{Cut}\left(\frac{p_1}{q_1}\oplus\frac{p_2}{q_2}\right)=\mathsf{Cut}\left(\frac{p_2}{q_2}\right)xy\mathsf{Cut}\left(\frac{p_1}{q_1}\right)$$
\end{enumerate}
\end{remk}

\subsection{Left codes, right codes and simple closed curves}
Let $X=X(\sigma, \tau)$ be a given square-tiled surface. Apply proposition 3.1.1 and Theorem 4.1.1, we have some algebraic properties about left codes and right codes. Moreover, these properties help us reduce the complexity of the computation of left and right codes. To make a summary, we have the following proposition:
\begin{prop} Let $\frac{P_1}{q_1}<_n\frac{p_2}{q_2}$ be any two rational numbers, then
\begin{enumerate}
\item $\mathsf{Code}^L\left(\frac{p_1}{q_1}\oplus\frac{p_2}{q_2}\right)= \mathsf{Code}^L\left(\frac{p_1}{q_1}\right)\mathsf{Code}^L\left(\frac{p_2}{q_2}\right).$
\item $\mathsf{Code}^R\left(\frac{p_1}{q_1}\oplus\frac{p_2}{q_2}\right)=\mathsf{Code}^R\left(\frac{p_2}{q_2}\right)\mathsf{Code}^R\left(\frac{p_1}{q_1}\right).$
\item $\mathsf{Code}^R\left(\frac{p_2}{q_2}\right)=\mathsf{Code}^L\left(\frac{p_1}{q_1}\right)\mathsf{Code}^L\left(\frac{p_2}{q_2}\right)\mathsf{Code}^L\left(\frac{p_1}{q_1}\right)^{-1}.$
\end{enumerate}
\end{prop}

In the rest of this section, we use continued fraction expansions to express rational numbers. Moreover, since the reflection with respect to the line $x=y$ will send any straight line with the slope $r>1$ to the straight line with the slope $\frac{1}{r}$, we only consider the rational number $r\in [0,1].$  We also ask for that for any rational number $r=[a_1,a_2,\ldots,a_k]\in [0,1)$, the positive integer $a_k$ is larger than 1. (If $a_k=1$, the continued fraction expansion is presented by $[a_1,a_2,\ldots,a_{k-1}+1]$.) A fundamental result is that for a rational number $r=[a_1,a_2,\ldots,a_k]$, its neighbors are $r'=[a_1,a_2,\ldots,a_k-1]$ and $r''=[a_1,a_2,\ldots,a_{k-1}]$. If $k$ is even, $r'<_nr<_nr''$ and therefore $r=r'\oplus r'';$ if $k$ is odd, $r$ is equal to $r''\oplus r'.$

Suppose the left code of a rational number $r$ contains a 1-cycle, the cylinder decomposition of the saddle connections with the slope $r$ must has a cylinder with surface area 1 and its left boundary consists of a saddle connection and contains only one marked point. Thus, we have the following definition:

\begin{defi} A rational number $[a_1,a_2, \ldots,a_n]$ is called a simple closed curve (abbrev. scc) if its left code contains a 1-cycle. Moreover, we say $[a_1,a_2, \ldots,a_n]$ is a scc at $k$ if $\mathsf{code}^L[a_1,a_2, \ldots,a_n](k)$ equals $k.$
\end{defi}

Applying the following property and corollaries, we not only construct a scc from some left codes, but also describe the distribution of some scc's at $m$.
\begin{prop} For a fixed square-tiled surface $X=X(\sigma,\tau)$ where $\sigma$ and $\tau$ are in $S_n.$
\begin{enumerate}
\item If $\mathsf{Code^L}([a_1,a_2,\ldots,a_m])$ is an n-cycle $(x_1,x_2,\ldots,x_n)\in{S_n}$, for any $k\in\{1,2,\ldots,n\}$ there exists a $t_k>0$ such that $[a_1,a_2,\ldots,a_m,t_k]$ is a scc at k.
\item If one of $[a_1,a_2,\ldots,a_k]$ and $[a_1,a_2,\ldots,a_{k+1}-1]$ is a scc at m, there exists a $t>0$ such that $[a_1,a_2,\ldots,a_{k+1},t]$ or $[a_1,a_2,\ldots,a_{k+1}-1,1,t]$ is also a scc at $m$.
\end{enumerate}
\end{prop}
\begin{proof} (1). Assume $m=2d+1$ and $\mathsf{Code^L}([a_1,a_2,\ldots a_{2d}])(k)=l$. Because $\mathsf{Code^L}([a_1,a_2,\ldots,a_{2d+1}])=(x_1,x_2,\ldots,x_n)\in{S_n}$, there exists a $t_k>0$ such that $\mathsf{Code^L}([a_1,a_2,\ldots,a_{2d+1}])^{t_k}(l)=k$. Due to the fact that $$\mathsf{Code^L}([a_1,a_2,\ldots,a_{2d+1},t_k])=\mathsf{Code^L}([a_1,a_2,\ldots a_{2d}])\mathsf{Code^L}([a_1,a_2,\ldots,a_{2d+1}])^{t_k},$$ we have $\mathsf{Code^L}([a_1,a_2,\ldots,a_{2d+1},t_k])(k)=k.$
Therefore, $[a_1,a_2,\ldots,a_{2d+1},t_k]$ is a scc at $k.$ For the case that $m=2d$, the proof is similar.\\

(2). Suppose $k+1$ is odd. Thus $[a_1,a_2,\ldots,a_k]<_n[a_1,a_2,\ldots,a_{k+1}]<_n[a_1,a_2,\ldots,a_{k+1}-1].$ For any positive integer $t,$ we have
\begin{align*}
&\mathsf{Code^L}([a_1,a_2,\ldots,a_{k+1},t])=\\&\mathsf{Code^L}([a_1,a_2,\ldots,a_k])\mathsf{Code^L}([a_1,a_2,\ldots,a_{k+1}])^t
\end{align*}
and
\begin{align*}
&\mathsf{Code^L}([a_1,a_2,\ldots,a_{k+1}-1,1,t])=\\&\mathsf{Code^L}([a_1,a_2,\ldots,a_{k+1}])^t\mathsf{Code^L}([a_1,a_2,\ldots,a_{k+1}-1]).
\end{align*}
 Since$[a_1,a_2,\ldots,a_k]$ is a scc at $m$, let $t$ be the order of $\mathsf{Code^L}([a_1,a_2,\ldots,a_{k+1}])$ and then we have a ssc $[a_1,a_2,\ldots,a_{k+1},t]$ at $m.$ Similarly, if $[a_1,a_2,\ldots,a_{k+1}-1]$ is a ssc and $t$ is the order of $\mathsf{Code^L}([a_1,a_2,\ldots,a_{k+1}]$, $[a_1,a_2,\ldots,a_{k+1}-1,1,t])$ is a ssc. For the case that $k+1$ is even, using the above approach, we get the desired result.
\end{proof}

Using similar tricks in the proof of proposition 4.2.2, the following corollaries are obvious:
\begin{cor}
For the left code $[a_1,a_2,\ldots,a_k]$, let $s$ be the order of $[a_1,a_2,\ldots,a_{k-1}].$ Then for any positive integer $m$, $$[a_1,a_2,\ldots,a_k+sm]=[a_1,a_2,\ldots,a_k].$$
\end{cor}
\begin{cor}
Let $\Omega_s:=\{i:[a_1,a_2,\ldots,a_k,i]\mbox{ is a scc at }s.\}.$ If $\Omega_s\neq\phi,$ there exist two numbers $u,v\in\{2,3,\ldots,Area(X)\}$ such that $$\Omega_s=\{u+vt:t\geq 0\}.$$
\end{cor}
\subsection{Left codes and Tori}
By the following propositions and corollaries, we can briefly understand the left codes of a square-tiled surface $X$.
\begin{prop}
If $[a_1,a_2,\ldots,a_k]$ and $[a_1,a_2, \ldots,a_k,a_{k+1}]$ are both ssc's at $m,$ $X$ contains a torus.
\end{prop}
\begin{proof} Since
\begin{align*}
&\mathsf{Code}^L\left([a_1,a_2,\ldots,a_k,a_{k+1}]\right)\\
&=\left\{\begin{array}{cc}
\mathsf{Code}^L\left([a_1,a_2,\ldots,a_{k}]\right)^{a_{k+1}}\mathsf{Code}^L\left([a_1,a_2,\ldots,a_{k-1}]\right) & \mbox{ when }k>1\mbox{ is even} \\
\mathsf{Code}^L\left(\left[a_1,a_2,\ldots, a_{k-1}\right]\right)\mathsf{Code}^L\left([a_1,a_2,\ldots,a_{k}]\right)^{a_{k+1}} & \mbox{ when }k>1\mbox{ is odd,}
\end{array}\right.
\end{align*}

these imply that both $[a_1,a_2,\ldots,a_k]$ and $[a_1,a_2,\ldots,{a_{k-1}}]$ are ssc's at $m.$ Inductively, we have the conclusion that both $[a_1]$ and $[a_{1}+1]$ are scc's at $m$. This implies that $\sigma$ fixes $m:$
\begin{align*} m=&\mathsf{Code}^L\left([a_1+1]\right)\mathsf{Code}^L\left([a_1]\right)^{-1}(m)=(\sigma^{a_1}\tau\sigma)(\sigma^{a_1-1}\tau\sigma)^{-1}(m)\\=&(\sigma^{a_1}\tau\sigma)(\sigma^{-1}\tau^{-1}\sigma^{1-a_1})(m)=\sigma(m)\end{align*}
Moreover, $\tau$ also fixes $m$ since $\mathsf{Code}^L\left([a_1+1]\right)=\sigma^{a_1}\tau\sigma$ and therefore $X$ contains a torus.\end{proof}
\begin{cor}
If both $[a_1,a_2,\ldots,a_k]$ and $[a_1,a_2,\ldots,a_k-1]$ are sccs at $m\leq n,$ then $X$ contains a torus.
\end{cor}
\begin{proof} Since
\begin{align*}&\mathsf{Code}^L\left([a_1,a_2,\ldots,a_k]\right)\\&=\left\{\begin{array}{lll}
\mathsf{Code}^L\left(\left[a_1,a_2,\ldots,a_k-1\right]\right)\mathsf{Code}^L\left([a_1,a_2,\ldots,a_{k-1}]\right) & \mbox{ when }k>1\mbox{ is even,} \\
\mathsf{Code}^L\left(\left[a_1,a_2,\ldots,a_{k-1}\right]\right)\mathsf{Code}^L\left([a_1,a_2,\ldots,a_{k}-1]\right) & \mbox{ when }k>1\mbox{ is odd,} \\
\mathsf{Code}^L\left(\left[a_1,a_2-1\right]\right)\mathsf{Code}^L\left([a_1]\right) & \mbox{ otherwise, }
\end{array} \right.\end{align*}
by the conclusion of proposition 4.3.1, we get the desired result.\end{proof}
\begin{prop}
If $[a_1,a_2,\ldots,a_k]$ and $[a_1,a_2,\ldots,a_{k-1}]$ are elements in a cyclic group $<\omega>\subset{S_n}$, $X$ is the disjoint union of tori.
\end{prop}
\begin{proof}
Suppose $k>2$ is an even number. Since $$\mathsf{Code}^L\left([a_1,a_2,\ldots,a_k]\right)=\mathsf{Code}^L\left([a_1,a_2,\ldots,a_{k-2}]\right)\mathsf{Code}^L\left([a_1,a_2,\ldots,a_{k-1}]\right)^{a_k},$$ $\mathsf{Code}^L\left([a_1,a_2,\ldots,a_{k-2}]\right)$ is also in the group $<\omega>.$   Inductively, the left codes of $[a_1]$ and $[a_1+1]$ are also in $<\omega>.$ Since $\mathsf{Code}^L\left([a_1+1]\right)=\sigma\mathsf{Code}^L\left([a_1]\right)$, this implies that $\sigma=\omega^{\alpha}$ and $\tau=\omega^{\beta}.$ These faces show that $\sigma\tau=\tau\sigma$ and thus all marked points are not critical points. When $k>2$ is an odd number, by the fact that $$\mathsf{Code}^L\left([a_1,a_2,\ldots,a_k]\right)=\mathsf{Code}^L\left([a_1,a_2,\ldots,a_{k-1}]\right)^{a_k}\mathsf{Code}^L\left([a_1,a_2,\ldots,a_{k-2}]\right) $$the proof is also similar to the above and thus the result follows.\end{proof}
\begin{cor}
Suppose that $0\leq r'<_{n}r''\leq 1$ are rational numbers and $r:=r'\oplus r''$. If any two elements in $\{\mathsf{Code}^L\left(r\right),\mathsf{Code}^L\left(r'\right),\mathsf{Code}^L\left(r''\right)\}$ are in a cyclic subgroup of $S_n$, the surface $X$ is the disjoint union of tori.
\end{cor}
\begin{cor}
For any connected square-tiled surface $X$ with genus $>1$, the left code of any rational number is not $(1)\in S_n.$
\end{cor}

\section{The closed system of square-tiled surfaces}
In Section 4.1, we used the Farey addition to compute the cutting sequence of a given line segment with rational slope. The Farey addition reduced the complexity of the computation of cutting sequences and record all left codes more systematically. In this section, we introduce the closed system of a given square-tiled surface. It records all left codes of $X$ and the relations between these permutation induced by the Farey addition. We will then discuss the  $SL_2^{+}(\mathbb{Z})$ action on the closed system of $X$. This gives a group structure on the collection of left codes of matrices in $SL_2^{+}(\mathbb{Z})$ and $Veech(X)$.

\subsection{The closed systems of square-tiled surfaces}
Let $X=X(\sigma,\tau)$ be a given square-tiled surface and suppose that $\frac{p}{q}<_n\frac{p_1}{q_1}$, $\omega=\mathsf{Code}^L\left(\frac{p}{q}\right)$ and $\omega_1=\mathsf{Code}^L\left(\frac{p_1}{q_1}\right)$. Note that for each positive integer $k$, we have $$\frac{p}{q}<_n\ldots<_n\frac{kp+p_1}{kq+q_1}<_n\ldots<_n\frac{p+p_1}{q+q_1}<_n\frac{p+kp_1}{q+kq_1}<_n\frac{p_1}{q_1}$$ and thus  $\mathsf{Code}^L(\frac{kp+p_1}{kq+q_1})=\omega^k\omega_1.$ If the order of $\omega$ equals $m$, $$\mathsf{Code}^L\left(\frac{kp+p_1}{kq+q_1}\right)=\mathsf{Code}^L\left(\frac{[k]_mp+p_1}{[k]_mq+q_1}\right).$$
In order to record all possible left codes of the rational numbers $\frac{kp+p_1}{kq+q_1}$, we define the ring diagram generated by two permutations as following:
\begin{defi} Let $u$ and $v$ be two elements in $S_n$ and $k+1$ be the order of $v.$ The ring diagram $B(v,u)$ is  shown in figure 4. Moreover, $v$ is called the center and each $v^mu$ is called a vertex of the ring diagram $B(v,u).$
\end{defi}
\begin{figure}[h]
\centering
\includegraphics[width=0.3\textwidth]{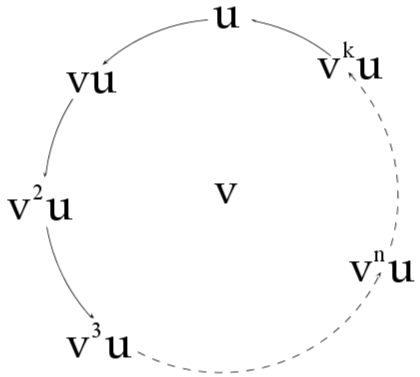}
\caption{The ring diagram $B(v,u)$}
\end{figure}
Obviously, for any integer $n$, the ring diagrams $B(u,v)$ and $B(u,u^nv)$ are the same.

\begin{exam}
Assume that $X=X((2,3),(1,2,4))$. Since $\mathsf{Code}^L\left(\frac{1}{1}\right)=(1,3,2,4)$ and $\mathsf{Code}^L\left(\frac{1}{0}\right)=(1,3,4),$ we have  $$\left\{\mathsf{Code}^L\left(\frac{\displaystyle k+1}{\displaystyle k}\right):k\in\mathbb{N}\right\}=
\{(1,3,4),(1,4,3,2),(1,2,3),(2,4)\}. $$ Figure 5 shows the ring diagram $B(\mathsf{Code}^L\left(\frac{1}{1}\right),\mathsf{Code}^L\left(\frac{1}{0}\right))$.
\begin{figure}[h]
\centering
\includegraphics[width=0.3\textwidth]{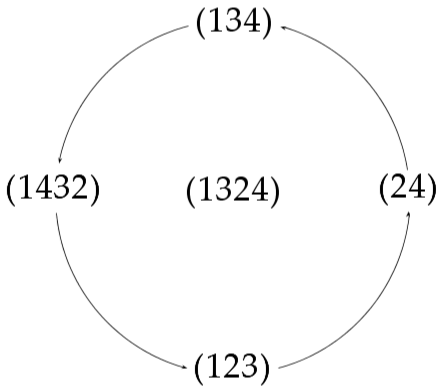}
\caption{The ring diagram $B((1324),(134))$}
\end{figure}
\end{exam}

\begin{defi}[The closed system of square-tiled surfaces]
Let $X=X(\sigma,\tau)$ be a square-tiled surface. The closed system of $X$, denoted by $\Omega_{(\sigma,\tau)}$ (or $\Omega_X$), is the collection of all ring diagrams $B\left(\mathsf{Code}^L\left(\frac{p}{q}\right),\mathsf{Code}^L\left(\frac{p_1}{q_1}\right)\right)$ where $\frac{p}{q}<_n\frac{p_1}{q_1}$ are two rational numbers. Moreover, for any $\omega\in S_n$,  the set $\omega\Omega_{X}\omega^{-1}$ is the collection of   $ B\left(\omega \mathsf{Code}^L\left(\frac{p}{q}\right)\omega^{-1},\omega \mathsf{Code}^L\left(\frac{p_1}{q_1}\right)\omega^{-1}\right).$
\end{defi}

We can obtain the closed system $\Omega_{(\sigma,\tau)}$ of $X(\sigma,\tau)$ by the following method: Let $A_1:=\{B(\sigma,\sigma^{-1}\tau\sigma)\}$. Suppose $A_k$ has been defined, then the set $A_{k+1}$ is defined by $$\left\{B(\alpha^{m+1}\beta,\alpha^{m}\beta):\alpha^{m+1}\beta\mbox{ and }\alpha^{m}\beta\ \mbox{ are vertices of }B(\alpha,\beta)\in{A_{k}\setminus A_{j},\forall\mbox{}j<k }\right\}.$$ Then the  closed system $\Omega_{(\sigma,\tau)}$ is the set $\bigcup_{j} A_j$.

Since $S_n$ is a finite group, the closed system generated by any two permutation elements is a finite union of ring diagrams. Next, we want to give some properties of the closed system of a surface $X=X(\sigma,\tau).$

\begin{prop}
Consider the square-tiled surface $X=X(\sigma.\tau)$. Let $\alpha^k\beta$ and $\alpha^{k+1}\beta$ be vertices in $B(\alpha,\beta)\in\Omega_{(\sigma,\tau)}$, then for the square-tiled surface $Y:=X(\alpha^{k+1}\beta,\alpha^{-1})$ $$B(\alpha,\beta)\in \Omega_{(\alpha^{k+1}\beta,\alpha^{-1})}.$$
\end{prop}
\begin{proof} Note that $\mathsf{Code}_Y^L\left(\frac{0}{1}\right)=\alpha^{k+1}\beta$ and $\mathsf{Code}_Y^L\left(\frac{1}{1}\right)=\alpha^k\beta$. For any positive integer $k$, we have $\frac{k-1}{k}<_n\frac{k}{k+1}<_n\frac{1}{1}$ and thus $$\mathsf{Code}_Y^L\left(\frac{k}{k+1}\right)=\mathsf{Code}_Y^L\left(\frac{k-1}{k}\right)\mathsf{Code}_Y^L\left(\frac{1}{1}\right).$$
Let $m$ be the order of $\mathsf{Code}_Y^L\left(\frac{1}{1}\right).$ Then we have
\begin{align*}
\mathsf{Code}_Y^L\left(\frac{m}{m+1}\right)&=\mathsf{Code}_Y^L\left(\frac{m-1}{m}\right)\mathsf{Code}_Y^L\left(\frac{1}{1}\right)\\
                                           &=\mathsf{Code}_Y^L\left(\frac{m-2}{m-1}\right)\mathsf{Code}_Y^L\left(\frac{1}{1}\right)^2\\
                                           &\vdots\\
                                           &=\mathsf{Code}_Y^L\left(\frac{0}{1}\right)\mathsf{Code}_Y^L\left(\frac{1}{1}\right)^m=\mathsf{Code}_Y^L\left(\frac{0}{1}\right)=\alpha^{k+1}\beta.
\end{align*}
This implies that $\mathsf{Code}_Y^L\left(\frac{m-1}{m}\right)=\alpha$. Since $\frac{m-1}{m}<_n\frac{1}{1}$, we have
$$B(\alpha,\alpha^{k}\beta)=B\left(\mathsf{Code}_Y^L\left(\frac{m-1}{m}\right),\mathsf{Code}_Y^L\left(\frac{1}{1}\right)\right)\subset\Omega_{(\alpha^{k+1}\beta,\alpha^{-1})}.$$ Thus the desired result follows because $B(\alpha,\alpha^k\beta)=B(\alpha,\beta).$ \end{proof}
\begin{cor}
$$\Omega_{(\alpha^{k+1}\beta,\alpha^{-1})}=\Omega_{(\alpha,\beta\alpha^{-1})}.$$
\end{cor}

\begin{thm}
Suppose $B(\gamma,\delta)\in\Omega_{(\sigma,\tau)},$ then $$\Omega_{(\sigma,\tau)}=\Omega_{(\gamma,\delta\gamma^{-1})}.$$
\end{thm}
\begin{proof} Suppose $\Omega_{(\sigma,\tau)}=\bigcup_i A_i$ and $B(\gamma,\delta)\in A_j,$ then $\gamma$ and $\delta$ must be equal to $\alpha^{k+1}\beta$ and $\alpha^{k}\beta$ respectively. Thus $B(\alpha^{k+1}\beta,\alpha^{k}\beta)\in A_j$ and by corollary  5.1.1, we have $\Omega_{(\gamma,\delta\gamma^{-1})}=\Omega_{(\alpha^{k+1}\beta,\alpha^{-1})}=\Omega_{(\alpha,\beta\alpha^{-1})}.$ Note that $B(\alpha,\beta)\in A_{j-1}$ and thus there exist $\alpha_1$ and $\beta_1$ such that $\alpha=\alpha_1^{k_{1}+1}\beta_1$ and $\beta=\alpha_1^{k_{1}}\beta_1$. Applying corollary 5.1.1 again, we have $\Omega_{(\alpha,\beta\alpha^{-1})}=\Omega_{(\alpha_1^{k+1}\beta_1,\alpha_1^{-1})}=\Omega_{(\alpha_1,\beta_1\alpha_{1}^{-1})}$ and thus $\Omega_{(\gamma,\delta\gamma^{-1})}=\Omega_{(\alpha_1,\beta_1\alpha_1^{-1})}.$

Repeating this argument, we conclude that  $$\Omega_{(\gamma,\delta\gamma^{-1})}=\Omega_{(\alpha^{k+1}\beta,\alpha^{-1})}=\Omega_{(\alpha,\beta\alpha^{-1})}=\Omega_{(\alpha_1,\beta_1\alpha_{1}^{-1})}=\ldots=\Omega_{(\sigma,\tau)}.$$
\end{proof}

\begin{remk} Suppose that the left codes of two rational numbers $\frac{p}{q}<_n\frac{p_1}{q_1}$ are $\omega_1$ and $\omega_2$ respectively. Theorem 5.1.1 shows that since the ring diagram $B\left(\mathsf{Code}^L\left(\frac{p}{q}\right),\mathsf{Code}^L\left(\frac{p+p_1}{q+q_1}\right)\right)=B(\omega_1,\omega_1\omega_2)$ in $\Omega_X$,
the closed system of $X$ and $X(\omega_1,\omega_1\omega_2\omega_1^{-1})$ are equal. Moreover, by the fact that $X(\omega_1,\omega_1\omega_2\omega_1^{-1})=X(\omega_1\omega_1\omega_1^{-1},\omega_1\omega_2\omega_1^{-1})=X(\omega_1,\omega_2)$, we obtain the following corollary.
\end{remk}
\begin{cor}
Let $\frac{p}{q},$ $\frac{p_1}{q_1},$ $\omega_1$ and $\omega_2$ be defined in remark 5.1.1. Then $$\Omega_X=\omega_1\Omega_{(\omega_1,\omega_2)}\omega_1^{-1}.$$
\end{cor}
\begin{prop}
Suppose that $B(\omega_1,\omega_2)$  is in the closed system $\Omega_X,$ then the surfaces $X$ and $Y=X(\omega_1,\omega_2\omega_1^{-1})$ are in the same stratum; that is, $X$ and $Y\in\mathcal{H}(a_1,a_2,\ldots,a_k).$
\end{prop}
\begin{proof} Since the closed system $\Omega_{(\sigma,\tau)}$ is the set $\bigcup_j A_j$, by the definition of each $A_i$, we only prove the followings:
\begin{enumerate}
  \item For any ring diagram $B(\omega_1,\omega_2)=B(\omega_1,\omega_{1}^{k}\omega_2)$ in $A_j,$ $X(\omega_1,\omega_2\omega_{1}^{-1})$ and $X(\omega_1,(\omega_{1}^{k}\omega_2)\omega_{1}^{-1})$ are in the same stratum.
  \item Suppose that the ring diagrams $B(\omega_1,\omega_2)\in A_j$ and $B(\omega_1^{k+1}\omega_2,\omega_1^{k}\omega_2)\in A_{j+1}$, then the surfaces $X(\omega_1,\omega_2\omega_{1}^{-1})$ and $X(\omega_1^{k+1}\omega_2,(\omega_{1}^{k}\omega_2)(\omega_{1}^{k+1}\omega_2)^{-1})( =X(\omega_1^{k+1}\omega_2,\omega_1^{-1}))$ are in the same stratum.
  \end{enumerate}

  For the first case, it is easy to get by the direct computation:
  $$\omega_1^{-1}(\omega_2\omega_{1}^{-1})^{-1}\omega_1(\omega_2\omega_1^{-1})=\omega_{2}^{-1}\omega_{1}\omega_{2}\omega_{1}^{-1}$$ and $$\omega_{1}^{-1}(\omega_{1}^{k}\omega_2)^{-1}\omega_1(\omega_{1}^{k}\omega_2)=\omega_1^{-1}\omega_{2}^{-1}\omega_{1}\omega_{2}.$$ (Note that these permutations are conjugate to each other.)\\
  For the second case, we also do the same computations and then get the desired result. To sum up, by the structure of the closed system of $X$,  the result follows.
\end{proof}
\begin{remk}
Define $S^+(X)$ to be the set
$$\left\{X(\omega_1,\omega_2)\mbox{ }|\mbox{ }\omega_1 \mbox{ and }\omega_2 \mbox{ are left codes of any rational numbers } 0\leq\frac{p}{q}<_n\frac{p_1}{q_1}\leq\infty\right\}.$$
Corollary 5.1.2 and proposition 5.1.2 indicate that up to conjugations, the closed system of any surface $Y\in{S^+(X)}$ is equal to $\Omega_X$ and $S^+(X)$ is a finite subset of $\mathcal{H}(a_1,a_2,\ldots,a_k).$ \end{remk}
Let $X_{90k}$ be the sqrare-tiled surface obtained by rotating $X$ $90k$ degrees clockwise. Then for $k=0,1,2\mbox{ and }3$, we have
\begin{align*}
&X=X_0=X(\sigma,\tau),\mbox{ }X_{90}=X(\tau,\sigma^{-1}),\\&X_{180}=X(\sigma^{-1},\tau^{-1}),\mbox{ }X_{270}=X(\tau^{-1},\sigma).
\end{align*}
\begin{prop}
Let $X=X(\sigma,\tau)$, then closed systems of $X$ and $X_{90}$ are conjugate to each other. Moreover, $$\Omega_X=\sigma^{-1}\Omega_{X_{90}} \sigma.$$
\end{prop}
\begin{proof}
Because $B(\sigma^{-1}\tau\sigma,\sigma^{-2}\tau\sigma)\in\Omega_X=\Omega_{(\sigma,\tau)}$, by the proposition 5.1.2, the closed system $\Omega_X$ is equal to $\Omega_{(\sigma^{-1}\tau\sigma,\sigma^{-1})}$. Since  $\Omega_{(\sigma^{-1}\tau\sigma,\sigma^{-1})}=\sigma^{-1}\Omega_{(\tau,\sigma^{-1})}\sigma=\sigma^{-1}\Omega_{90}\sigma,$ we get the desired result.
\end{proof}

\begin{cor}
Let $X=X(\sigma,\tau)$ and $\Theta=\sigma^{-1}\tau^{-1}\sigma\tau$, then $$\Omega_X=\Theta\Omega_X\Theta^{-1}.$$
\end{cor}
\begin{proof} By  Proposition 5.1.4, we have the facts that $\Omega_X=\sigma^{-1}\Omega_{X_{90}}\sigma$, $\Omega_{X_{90}}=\tau^{-1}\Omega_{X_{180}}\tau,$ $\Omega_{X_{180}}=\sigma \Omega_{X_{270}}\sigma^{-1}$ and $\Omega_{X_{270}}=\tau \Omega_{X}\tau^{-1}$. Therefore,
\begin{align*}
\Theta\Omega_X\Theta^{-1}&=\sigma^{-1}\tau^{-1}\sigma\tau\Omega_X\tau^{-1}\sigma^{-1}\tau\sigma\\
                         &=\sigma^{-1}\tau^{-1}\sigma\Omega_{X_{270}}\sigma^{-1}\tau\sigma\\
                         &=\sigma^{-1}\tau^{-1}\Omega_{X_{180}}\tau\sigma\\
                         &=\sigma^{-1}\Omega_{X_{90}}\sigma\\
                         &=\Omega_X.
\end{align*}

\end{proof}

\subsection{The action of $SL_2^{+}(\mathbb{Z})$ on closed systems}
\begin{defi}
$SL_2^{+}(\mathbb{Z})$ is a subset of $SL_2(\mathbb{Z})$ and defined by $$\left\{\left[\begin{array}{cc}
a_{11} & a_{12} \\
a_{21} & a_{22}
\end{array} \right]\in{SL_2(\mathbb{Z})}\mid a_{ij}\geq 0 \mbox{ for all }i,j\right\}.$$
\end{defi}
\begin{remk} $\mbox{ }$
\begin{enumerate}
\item Let $X=X(\sigma,\tau)$ be a square-tiled surface, then the derivatives of the Dehn twists with respect to the horizontal and vertical saddle cylinder decompositions  are $$H=\left[\begin{array}{cc}
1 & h \\
0 & 1
\end{array} \right]\mbox{, }V=\left[\begin{array}{cc}
1 & 0 \\
v & 1
\end{array}\right]$$
where $h$ and $v$ are positive integers and therefore in $SL^{+}_2(\mathbb{Z}).$
\item Define $L$ and $R$ to be the matrices: $$L:=\left[\begin{array}{cc}
1 & 1 \\
0 & 1
\end{array} \right]\mbox{ }R:=\left[\begin{array}{cc}
1 & 0 \\
1 & 1
\end{array}\right].$$
$SL_2(\mathbb{Z})$ is generated by $L$ and $R.$ Moreover, these two matrices generate $SL^{+}_2(\mathbb{Z})$ ``positively"; that is, for every element $A\neq I$ in $SL^{+}_2(\mathbb{Z})$, $A$ can be represented by $L^{n_1}R^{m_1}\ldots L^{n_k}R^{m_k}$ where $n_i$ and $m_j$ are non negative integers.
\item A Farey pair $(\frac{p}{q},\frac{p_1}{q_1})$ is a pair of  two rational numbers $\frac{p}{q}<_n\frac{p_1}{q_1}$. There is a bijection between all Farey pairs $SL^{+}_2(\mathbb{Z})$:
$$\left(\frac{p}{q},\frac{p_1}{q_1}\right)\longleftrightarrow \left[\begin{array}{cc}
q & q_1 \\
p & p_1
\end{array} \right].$$
\end{enumerate}
\end{remk}
For convenience, we give the following definition:
\begin{defi}Let $X=X(\sigma,\tau)$ be a square titled surface and
$A=\left[\begin{array}{cc}
a & c \\
b & d
\end{array} \right]$ be an element in $SL_2^{+}(\mathbb{Z})$, then the left code of $A$ is defined by $$\mathsf{Code}^L(A):=\left(\mathsf{Code}^L\left(\frac{b}{a}\right),\mathsf{Code}^L\left(\frac{d}{c}\right)\right).$$ Moreover, suppose that $\mathsf{Code}^L(A)=(\omega_1,\omega_2)$ and $\alpha\in{S_n}$. Then $\alpha\mathsf{Code}^L(A)\alpha^{-1}$ is the pair of permutations $(\alpha \omega_1\alpha^{-1},\alpha \omega_2\alpha^{-1}).$
\end{defi}
\begin{prop}
Let $A$ and $A'$ be in $SL_2^{+}(\mathbb{Z})$. Suppose that for some $\alpha\in  S_n$,  $\mathsf{Code}^L(A)=\alpha \mathsf{Code}^L(A')\alpha^{-1}$. Then for any non negative integer $k$, we have $$\begin{array}{cc}
\mathsf{Code}^L(AL^k)=\alpha\mathsf{Code}^L(A'L^k)\alpha^{-1},& \mathsf{Code}^L(AR^k)=\alpha\mathsf{Code}^L(A'R^k)\alpha^{-1}.
\end{array} $$
\end{prop}
\begin{proof} We first prove $\mathsf{Code}^L(AL^k)=\alpha\mathsf{Code}^L(A'L^k)\alpha^{-1}$ by induction. For $k=0$, by the assumption, we have  $$\mathsf{Code}^L(AL^0)=\mathsf{Code}^L(A)=\alpha\mathsf{Code}^L(A')\alpha^{-1}=\alpha\mathsf{Code}^L(A'L^0)\alpha^{-1}.$$ Assume that it is true  for $k=n,$ i.e., $\mathsf{Code}^L(AL^n)=\alpha\mathsf{Code}^L(A'L^n)\alpha^{-1}.$ When $k=n+1$, we may assume that $$\begin{array}{cc}
AL^n=\left[\begin{array}{cc}
a_1 & b_1 \\
a_2 & b_2
\end{array} \right], & A'L^n=\left[\begin{array}{cc}
a'_1 & b'_1 \\
a'_2 & b'_2
\end{array} \right]\end{array}$$ and therefore
\begin{align*}
\mathsf{Code}^L\left(AL^{n+1}\right)&=\mathsf{Code}^L\left(AL^n\cdot L\right)=\left(\mathsf{Code}^L\left(\frac{a_2}{a_1}\right),\mathsf{Code}^L\left(\frac{a_2+b_2}{a_1+b_1}\right)\right)\\
&=\left(\mathsf{Code}^L\left(\frac{a_2}{a_1}\right),\mathsf{Code}^L\left(\frac{a_2}{a_1}\right)\mathsf{Code}^L\left(\frac{b_2}{b_1}\right)\right)\\
&\mbox{ (by the induction hypothesis)}\\
&=\alpha \left(\mathsf{Code}^L\left(\frac{a'_2}{a'_1}\right),\mathsf{Code}^L\left(\frac{a'_2}{a'_1}\right)\mathsf{Code}^L\left(\frac{b'_2}{b'_1}\right)\right) \alpha^{-1}\\
&=\alpha\mathsf{Code}^L(A'L^{n+1})\alpha^{-1}.
\end{align*}
This proves that $\mathsf{Code}^L(AL^k)=\alpha\mathsf{Code}^L(A'L^k)\alpha^{-1}.$
The proof for $\mathsf{Code}^L(AR^k)=\alpha\mathsf{Code}^L(A'R^k)\alpha^{-1}$ is similar to the above.
\end{proof}
Since every element in $SL_2^{+}(\mathbb{Z})$ is generated by $L$ and $R$ positively, the following corollaries are obvious.
\begin{cor}
Let $A$, $B$ and $C$ be in $SL_2^{+}(\mathbb{Z})$ and $\mathsf{Code}^L(A)=\alpha\mathsf{Code}^L(B)\alpha^{-1}$ for some $\alpha\in S_n,$ then $$\mathsf{Code}^L(AC)=\alpha\mathsf{Code}^L(BC)\alpha^{-1}.$$
\end{cor}
\begin{cor}Let $\mathsf{Code}^L(A)=\alpha\mathsf{Code}^L(I)\alpha^{-1}$ for some $\alpha\in S_n,$ then $$\mathsf{Code}^L(A^k)=\alpha^{k}\mathsf{Code}^L(I)\alpha^{-k}.$$
\end{cor}
\begin{remk} The above proposition and corollaries imply that if $\mathsf{Code}^L(A)=u\mathsf{Code}^L(I)u^{-1}$ and $\mathsf{Code}^L(B)=v\mathsf{Code}^L(I)v^{-1}$, $\mathsf{Code}^L(AB)$ is equal to $uv\mathsf{Code}^L(I)u^{-1}v^{-1}.$ Especially, suppose that $A$ is the horizontal or vertical Dehn twist. Because $\mathsf{Code}^L(A)=\mathsf{Code}^L(I)$, we have  $\mathsf{Code}^L(AB)=\mathsf{Code}^L(BA)=\mathsf{Code}^L(B).$
\end{remk}

\subsection{The group structures of left codes}
Let $X=X(\sigma,\tau)$ be a square-tiled surface. Define the group $$G_X=\left\{\omega\in S_n\mid \omega\mathsf{Code}^L(I)\omega^{-1}=\mathsf{Code}^L(A)\mbox{ for some }A\in SL^+_2(\mathbb{Z}) \right\}$$ and its subgroup
$$ S_X=\left\{\omega\in{G_X}\mid \omega\mathsf{Code}^L(I)\omega^{-1}=\mathsf{Code}^L(I)\right\}.$$
Corollary 5.1.3 indicates that $G_X$ is a non trivial subgroup of $S_n.$ Moreover, if we define
$$Veech^{+}(X)=Veech(X)\cap SL^{+}_{2}(\mathbb{Z})$$
by Theorem 6.1.1 (see section 6.1), $G_X$ is equal to
$$\left\{\omega\in S_n\mid \omega\mathsf{Code}^L(I)\omega^{-1}=\mathsf{Code}^L(A)\mbox{ for some }A\in Veech^+(X) \right\}.$$
\begin{lemma} $S_X$ is a normal subgroup of $G$.
\end{lemma}
\begin{proof} Let $u\in G_X$ and $\omega\in S_X,$ and assume that $u\mathsf{Code}^L(I)u^{-1}=\mathsf{Code}^L(A)$ for some $A\in SL^+_2(\mathbb{Z})$. We want to show that $u\omega u^{-1}$ is in $S_X$.

By Corollary 5.2.2, $u^{-1}\mathsf{Code}^L(I)u=u^k\mathsf{Code}^L(I)u^{-k}=\mathsf{Code}^L(A^k)$ where the order of $u$ is $k+1.$ Therefore, we have
\begin{align}
u\omega u^{-1}\mathsf{Code}^L(I)u\omega^{-1}u^{-1}&=u\omega(u^{-1}\mathsf{Code}^L(I)u)\omega^{-1}u^{-1} \notag\\
&=u\omega\mathsf{Code}^L(A^k)\omega^{-1}u^{-1}\tag{1}
\end{align}

Since the left code of any rational number is a permutation which is equal to the product of  the left codes of $\frac{0}{1}$ and $\frac{1}{0}$, by the assumption that $\omega\mathsf{Code}^L(I)\omega^{-1}=\mathsf{Code}^L(I)$, $\omega\mathsf{Code}^L(A)\omega^{-1}$ is equal to $\mathsf{Code}^L(A)$ for any matrix $A\in{SL^+_2(\mathbb{Z})}$. This implies that
 (1) is equal to $u\mathsf{Code}L(A^k)u^{-1}=\mathsf{Code}^L(I)$ and the proof is complete.
\end{proof}

Since $S_X$ is a normal subgroup of $G_X,$ we can define the following map:
\begin{align*}
\Phi:G_X/S_X &\longrightarrow \mathsf{Code}^L(Veech^+(X))\\
[u]&\longrightarrow u\mathsf{Code}^L(I)u^{-1},
\end{align*}
Where $ \mathsf{Code}^L(Veech^+(X))=\{\mathsf{Code}^L(A)\mid A\in{Veech^+(X)}\}.$
\begin{remk} Since $S_X$ is a normal subgroup of $G_X$, $\Phi$ is well-defined and injective. Moreover, by Theorem 6.1.1, the map is surjective.
\end{remk}

For any $A$ and $B$ in $Veech^+(X)$, define $$\mathsf{Code}^L(A)\odot\mathsf{Code}^L(B)=\mathsf{Code}^L(AB).$$  Under this operation, we have the following proposition.
\begin{prop}
$(\mathsf{Code}^L(Veech^+(X)),\odot)$ is a group and $\Phi$ is a group isomorphism.
\end{prop}
\begin{proof} Clearly, $\mathsf{Code}^L(Veech^+(X))$ is closed under the binary operation $\odot$ and satisfies the associativity. It is easy to see that the identity is  $\mathsf{Code}^L(I).$ Now, let $\tau_A\in\{\sigma_A\mid \mathsf{Code}^L(A)=\sigma_A\mathsf{Code}^L(A)\sigma_A^{-1}\}\setminus\{I\}$ with the smallest order $n$, then
\begin{align*}
\mathsf{Code}^L(A)\odot\mathsf{Code}^L(A^{n-1})&=\mathsf{Code}^L(A^{n-1})\odot\mathsf{Code}^L(A)\\
&=\mathsf{Code}^L(A^n)=\tau_A^n\mathsf{Code}^L(I)\tau_A^{-n}\\
&=\mathsf{Code}^L(I).
\end{align*}

This implies that the inverse of $\mathsf{Code}^L(A)$ is $\mathsf{Code}^L(A^{n-1})$ and thus $(\mathsf{Code}^L(Veech^+(X)),\odot)$ is a group. Now, consider the following map
\begin{align*}
\phi:G_X &\longrightarrow \mathsf{Code}^L(Veech^+(X))\\
u&\longrightarrow u\mathsf{Code}^L(I)u^{-1}.
\end{align*}

Applying Remark 5.2.2, we have the conclusion that $\phi$ is a group homomorphism. Since the kernel of $\phi$ is $S_X$, the result follows.
\end{proof}

\section{The applications of closed systems}
\subsection{The Veech elements}
Let $X=X(\sigma,\tau)$ be a given square tiled surface tiled by the squares $Q_i$, $i=1,\ldots,n$. In this section, we will use the left codes to study the set $Veech^{+}(X).$

The main theorem is the following:
\begin{thm}
Let $A$ be an element in $SL_2^+(\mathbb{Z)}$ and $I$ be the identity matrix. Then $A\in Veech^+(X)$ if and only $\mathsf{Code}^L(A)$ is conjugate to $\mathsf{Code}^L(I).$
\end{thm}
\begin{proof} Suppose $A\in SL_2^+(\mathbb{Z)}$ with the 1st. (resp 2nd) column vector $(a,b)$ (resp. $(c,d)$). Let $X=X(\sigma,\tau)$ be a square-tiled surface and $\sigma=\sigma_1\sigma_2\ldots\sigma_j$ where $\sigma_i=(a_1^i,a_2^i\ldots,a^i_{k_i})$ is a $k_i-$cycle. This implies that $X$ is equal to $\bigcup_{i=1}^jH_i$ where $H_i$ is a horizontal cylinder and each $H_i$ is the union of the squares $Q_{a^i_m}$ for $m=1\ldots k_i.$

For any orientation-preserving affine diffeomorphism $f_A:X\rightarrow X$ with the derivative $A$, $f_A(Q_{a^i_m})$ is a parallelogram spanned by the vectors $(a,b)$ and $(c,d)$ and thus $f_A(H_i)$ is an Euclidean cylinder which is the union $\bigcup_{m=1}^{k_i}f_A(Q_{a^i_{m}}).$ Moreover, because $f_A$ keeps the orientation, the left boundaries of each $f(H_i)$ are equal to the images of the left boundaries of $H_i$ under $f_A$. The``south-west corner" of each $f_A(Q_{a^i_{m}})$ is also the south-west corner of some square $Q_{a^j_{s_m}}$. Let $\omega\in S_n$ defined by $\omega(a^j_{s_m})=a^i_{m},$ for all $i=1,\ldots,j$ and $m=1,\ldots,k_i.$ Then we have $\mathsf{Code}^L\left(\frac{b}{a}\right)=\omega\mathsf{Code}^L\left(\frac{0}{1}\right)\omega^{-1}.$ Applying the similar argument,
$\mathsf{Code}^L\left(\frac{d}{c}\right)$ equals $\omega\mathsf{Code}^L\left(\frac{1}{0}\right)\omega^{-1}$ and thus the result follows.

Conversely, suppose there is a permutation $\omega\in S_n$ such that  $\mathsf{Code}^L(A)=\omega^{-1}\mathsf{Code}^L(I)\omega$. Then the left code of $\frac{b}{a}$ and right code of $\frac{d}{c}$ are conjugate to $\sigma$ and $\tau$ respectively.

Since $X$ can be represented by the union of the parallelograms spanned by the vectors $(a,b)$ and $(c,d)$ and the ``south-west corner" of each parallelogram is also the south-west corner of some square, every parallelogram can be marked  by a number $k$ if its south-west corner is also that of the square marked by $k.$ This implies that $X$ is a parallelogram-tiled surface whose $(a,b)$ direction tiling is recorded by $\mathsf{Code}^L\left(\frac{b}{a}\right)$ and $(c,d)$ direction tiling is recorded by $\mathsf{Code}^R\left(\frac{d}{c}\right)$.

Now, let $Q_i\subset X$ be a square and $P_{\omega(i)}\subset X$ be a parallelogram spanned by the vectors $(a,b)$ and $(c,d)$. Define the linear map $T_i:Q_i\rightarrow P_{\omega(i)}$ by  $T_i(s_i+x)=p_i+Ax$ where $s_i$ and $p_i$ are the centers of $S_i$ and $P_{\omega(i)}$ respectively. Let $T:X\rightarrow X$ be the map defined by $T(p)=T_i(p)$ if $p\in{S_i}$ and this gives an affine diffeomorphism on $X$ with the derivative $A.$
\end{proof}
\begin{exam}[Eierlegende Wollmilchsau]
Let $$X=X((1,2,3,4)(5,6,7,8),(1,8,3,6)(2,7,4,5)).$$ Its Veech group is $SL_2(\mathbb{Z})$ and thus $Veech^{+}(X)=SL^+_2(\mathbb{Z})$. The followings are all possible left codes of $X$
\begin{align*}
A_1&=(1,2,3,4)(5,6,7,8),A_2=(1,4,3,2)(5,8,7,6),A_3=(1,5,3,7)(2,8,4,6),\\
A_4&=(1,7,3,5)(2,6,4,8),A_5=(1,6,3,8)(2,5,4,7),A_6=(1,8,3,6)(2,7,4,5).
\end{align*}
Since the order of each $A_j$ is 4 and $A_{2k-1}^{-1}=A_{2k}$ for $k=1,2,3, $ each pair $(A_i, A_j(\neq A_i^{-1}))$ is the left code of some matrices in $Veech^{+}(X).$ By theorem 5.1.1, we conclude that all of these pairs of left codes are conjugate to each other.
In fact, for any matrix $A\in SL^+_2(\mathbb{Z})$, there is an element $\omega_A$ in the group $\left<(1,2,3,4),(1,5,3,7),(1,6,3,8)\right>$ such that  $\mathsf{Code}^{L}\left(A\right)$ is equal to $\omega_A\mathsf{Code}^{L}\left(I\right)\omega_A^{-1}$.
\end{exam}
\begin{remk} Let $X=X(\sigma,\tau)$ be a swuare-tiled surface.
\begin{enumerate}
\item The derivatives of horizontal and vertical Dehn twists  of $X$ are Veech elements and their left codes are equal to the left code of $I.$
\item Suppose $A=\left[\begin{array}{cc}
a & c \\
b & d
\end{array} \right]\in{ SL^+_2(\mathbb{Z})}$ and let
\begin{align}
Y=&X\left(\mathsf{Code}_X^L\left(\frac{b}{a}\right),\mathsf{Code}_X^R\left(\frac{d}{c}\right)\right)\notag\\
 =&X\left(\mathsf{Code}_X^L\left(\frac{b}{a}\right),\mathsf{Code}_X^L\left(\frac{d}{c}\right)\right)\tag{Proposition 4.2.1}.
\end{align}
Then applying the similar argument in the proof of Theorem 6.1.1, there exists an affine diffeomorphism $T:Y\rightarrow X$ with $DT=A.$
\item Let $S(X)$ be the orbit of $X$ under the $SL_2(\mathbb{Z})$ action and suppose $Y=X(\alpha,\beta)\in S(X)$.If the derivative of  an affine diffeomorphisms $f:Y\rightarrow X$ is in $SL^+_2\left(\mathbb{Z}\right)$, then $(\alpha,\beta)$ is conjugate to $\left(\mathsf{Code}^L_X\left(\frac{a_{21}}{a_{11}}\right),\mathsf{Code}^R_X\left(\frac{a_{22}}{a_{12}}\right)\right).$
\end{enumerate}
\end{remk}

Recall that for the square-tiled surface $X=X(\sigma,\tau),$ the surface $X_{90}=X(\tau,\sigma^{-1})$ is obtained by rotating $X$ 90 degree clockwis. Applying Theorem 6.1.1, we have the following results:
\begin{prop}
Suppose $X=X(\sigma,\tau)$ and $A=\left[\begin{array}{cc}
a & b \\
c & d
\end{array} \right]\in SL^+_2(\mathbb{Z})$, then
$$\mathsf{Code}^L_X(A)\mbox{ is  conjugate to }\mathsf{Code}^L_{X_{90}}(I)\mbox{ if and only if   }\left[\begin{array}{cc}
-c & a \\
-d & b
\end{array} \right]\in Veech(X).$$
\end{prop}

\begin{cor}
Suppose $X=X(\sigma,\tau)$ and $A=\left[\begin{array}{cc}
a & b \\
c & d
\end{array} \right]\in SL^+_2(\mathbb{Z})$, then
$$\mathsf{Code}^L_X(A)\mbox{ is  conjugate to }\mathsf{Code}^L_{X_{180}}(I)\left(resp. \mathsf{Code}^L_{X_{270}}(I)\right)$$if and only if $$\left[\begin{array}{cc}
-a & -c \\
-b & -d
\end{array} \right]\left(resp. \left[\begin{array}{cc}
c & -a \\
d & -b
\end{array} \right]\right)\in Veech(X).$$
\end{cor}

\subsection{Orbits of $(X(\sigma,\tau))$ under the $SL_2(\mathbb{Z})$ action}

In this section, we will use the closed system of a given square-tiled surface to describe the orbit of  $X$ under the $SL_2(\mathbb{Z})$ action. Let $X=X(\sigma,\tau)$ be a square-tiled surface and consider the ring diagram $B(\gamma,\delta)\in \Omega_{(\sigma,\tau)}.$ Corollary 5.1.2 implies that $S^+(X)$ is equal to $S^+(X(\gamma,\delta)).$ Thus for any surface $Y\in S^+(X),$ we have $S^+(X)=S^+(Y).$

Assume $L:=\left[\begin{array}{cc}
1 & 1 \\
0 & 1
\end{array} \right]$ and $R:=\left[\begin{array}{cc}
1 & 0 \\
1 & 1
\end{array} \right]$. We define the map $l:S^(X)\rightarrow S^(X)$ (resp. $r:S^(X)\rightarrow S^(X)$) by the following: let $Y\in S^(X)$, then $l(X)$ (resp. r(X)) is the surface such that there is an affine diffeomorphism $f_Y:l(Y)\rightarrow Y$ (resp. $f_Y:l(Y)\rightarrow Y$) with the derivative $L$ (resp. R).

Applying Remark 6.1.1, the maps $l$ and $r$ are surjective and thus bijective. Since the group $SL_2(\mathbb{Z})$ is generated by $L$ and $R$, we can conclude that the orbit $S(X)$ of the surface $X$ under the $SL_2(\mathbb{Z})$ action is $S^+(X).$

To sum up the above results, we have the following proposition:
\begin{thm}
Let $X$ be a square-tiled surface, then
\begin{enumerate}
\item $S(X)=S^+(X)$.
\item The index of the Veech group in $SL_2(\mathbb{Z})$ is $|S^+(X)|.$
\end{enumerate}
\end{thm}


\begin{thebibliography}{9}



\bibitem[1]{} Pascal Hubert and Samuel Leli\`{e}vre. Prime arithmetic Teichm\"{u}ller discs in $\mathcal{H}(2)$. \textit{Israel Journal of mathematics}, \textbf{15}: , 281-321, 2006.

\bibitem[2]{} Pascal Hubert and Thomas Schmidt. An introduction to Veech surfaces. \textit{Handbook of dynamical systems. Vol. 1B, Elsevier B. V., Amsterdam, pp. 501–526,}

\bibitem[3]{}E.Gutkin. Billards on almost integrable polyhedral surface. \textit{Ergodic Theory and Dynamical Systems}, \textbf{4}:569-584, 1984.

\bibitem[4]{}E.Gutkin and C. Judge. The geometry and arithmetic of translation surfaces with applications to polygonal billiards. \textit{Mathematical Research Letter}, \textbf{3}:391-403, 1996.

\bibitem[5]{}E.Gutkin and C. Judge. Affine mappings of translation surfaces: geometry and arithmetic. \textit{Duck mathematical Journal}, \textbf{103}:191-213, 2000.

\bibitem[6]{}F. Herrlich, A.Kappes and G.  Schmith\"{u}sen. An origami of genus 2 with a translation. \textit{ariv:math.AG/0805.1865}.

\bibitem[7]{}F.Herrlich and G. Schmith\"{u}sen. On the boundary of Teichm\"{u}ller disks in Teichm\"{u}ller and Schottky space. \textit{Handbook of Teichmüller theory. Vol. I}, 293–349, IRMA Lect. Math. Theor. Phys., 11, Eur. Math. Soc., Zürich, 2007.

\bibitem[8]{}Pascal Hubert, An introduction to Veech surfaces. \textit{Handbook of dynamical systems. Vol. 1B}, 501–526, Elsevier B. V., Amsterdam, 2006.

\bibitem[9]{} M. Kontsevich and A Zorich. Connected components of the moduli spaces of Abelian differentials with prescribed singularities. \textit{Inventiones Mathematicae}, \textbf{153}:631-678, 2003.

\bibitem[10]{} C.T. McMullen. Teichm\"{u}ller curves in genus two: discriminant and spin. \textit{Math. Ann.}, \textbf{333}:87-130, 2005.

\bibitem[11]{} C.T. McMullen. Billiard and Teichm\"{u}ller curves on Hilbert modular surfaces. \textit{J. Amer. Math. Soc.} \textbf{16}:857–885 2003.

\bibitem[12]{}G. Schmith\"{u}sen. An Algorithm for finding the Veech group of an origami.
  \textit{Experimental Mathematics}, \textbf{13}:459-472, 2004.

\bibitem[13]{}G. Schmith\"{u}sen. Examples for Veech groups of origamis. \textit{The geometry of Riemann surfaces and Abelian varieties: III Iberoamerican Congress on Geometry in honor of Professor Sevín Recillas-Pishmish's 60th birthday, June 8-12, 2004, Salamanca, Spain.} American Mathematical Soc., 2006.


\bibitem[14]{} John Smillie and Corinna Ulcigrai. Symbolic coding for linear trajectories in the regular octagon. \textit{ariv:math.DS/0905.0871v1}.

\bibitem[15]{} W. A. Veech. Geometric realizations of hyperbolic curves. \textit{Chaos, Dynamics and Fractals.} Plemun, 1995.

\bibitem[16]{} W. A. Veech. Teichm\"{u}ller curves in moduli spaces, Eisenstein series and an application to triangular billiards. \textit{Inv. Math.}, \textbf{97}:553-583, 1989.

\bibitem[17]{} W. P. Thurston. On the geometry and dynamics of diffeomorphisms of surfaces. \textit{Bllentin of the american Mathematical Society(N.S.)},  \textbf{19}:417-431, 1998.

\bibitem[18]{} Anton Zorich. Flat Surfaces. \textit{ariv:math.DS/0609392v2}.

\end{thebibliography}
\end{document}